\documentclass[10pt,leqno]{amsart}
\usepackage[T1]{fontenc}
\usepackage[latin9]{inputenc}
\usepackage{amstext}
\usepackage{amsthm}
\usepackage{amssymb}

\makeatletter
\numberwithin{equation}{section}
\numberwithin{figure}{section}
\theoremstyle{plain}
\newtheorem{thm}{\protect\theoremname}
\theoremstyle{remark}
\newtheorem{rem}[thm]{\protect\remarkname}
\theoremstyle{definition}
\newtheorem{defn}[thm]{\protect\definitionname}
\theoremstyle{plain}
\newtheorem{lem}[thm]{\protect\lemmaname}
\theoremstyle{plain}
\newtheorem{cor}[thm]{\protect\corollaryname}
\theoremstyle{plain}
\newtheorem{prop}[thm]{\protect\propositionname}

\@ifundefined{date}{}{\date{}}

\providecommand{\lemmaname}{Lemma}
\providecommand{\remarkname}{Remark}
\providecommand{\theoremname}{Theorem}

\makeatother

\providecommand{\corollaryname}{Corollary}
\providecommand{\definitionname}{Definition}
\providecommand{\lemmaname}{Lemma}
\providecommand{\propositionname}{Proposition}
\providecommand{\remarkname}{Remark}
\providecommand{\theoremname}{Theorem}

\begin{document}
\title{Generalized spherical mean value operators on Euclidean space}
\author{Yasunori Okada and Hideshi Yamane}
\address{{\small{}Institute of Management and Information Technologies, Chiba
University}\\
{\small{}1-33 Yayoicho, Inage-ku, Chiba, 263-8522 Japan}}
\address{{\small{}Department of Mathematical Sciences, Kwansei Gakuin University
}\\
{\small{}2-1, Gakuen, Sanda, Hyogo 669-1337, Japan}}
\email{{\small{}okada@math.s.chiba-u.ac.jp, yamane@kwansei.ac.jp}}
\begin{abstract}
We consider the Neumann version of the spherical mean value operator
and its variants in the space of smooth functions, distributions and
compactly supported ones. Surjectivity and range characterization
issues are addressed from the viewpoint of convolution equations.
\end{abstract}

\thanks{The first author is supported by JSPS KAKENHI Grant Number 16K05170.}

\maketitle
\markboth{Okada-Yamane}{Generalized mean value operators}\tableofcontents

\section{Introduction}

Convolution equations are natural extensions of linear partial differential
equations with constant coefficients. Malgrange \cite{MR86990} 
in 1955 proved that the convolution equation $\mu*g=f$, where $\mu\ne0$
is an arbitrary analytic functional, has an entire solution $g$ for
any entire function $f$. Ehrenpreis \cite{MR119082} in 1956 studied
equations with distribution kernels on the spaces differentiable functions,
distributions, and real analytic functions. Such equations have been
studied by many authors on various situations, for example, on spaces
of holomorphic functions on convex domains, those with growth conditions
near the boundary, hyperfunctions, Fourier hyperfunctions, etc. See
H\"ormander \cite{MR141984}, Korobe\u{\i}nik \cite{MR0226148},
Kawai \cite{MR0298200}, Ishimura-Okada \cite{MR1294464}, Abanin-Ishimura-Khoi
\cite{MR2890341}, Langenbruch \cite{MR3066407} and the references
therein. Necessary or sufficient conditions for solvability are often
written in terms of the Fourier (or Laplace) transform of the kernel. 

Lim \cite{Lim} proved that the spherical mean value operators on
Euclidean and hyperbolic spaces are  surjective convolution operators.
He employed techniques devised by Ehrenpreis and H\"ormander. Christensen-Gonzarez-Kakehi
\cite{Kakehi} extended this result to the general case of noncompact
symmetric spaces. In \cite{Kakehi} and \cite{Lim}, the operators
involve Dirichet boundary values on spheres. So it is natural to consider
the Neumann case. In the present paper, we study the surjectivity
of the Neumann mean value operator on Euclidean space and higher order
variants. We show that they are surjective on the space of smooth
functions and that of distributions. Moreover, we characterize the
ranges in the case of compact supports.  

The proofs basically follow those in \cite{Lim}, but we have simplified
some parts of the arguments. 

Mean value operators are studied in various settings. See, for example,
Agranovsky et al. \cite{Agranovsky} and Antipov et al. \cite{Antipov}.
It may be possible to prove the Neumann versions of these results.

\section{Distributions, convolution and the Fourier transform}

Let $S(x,r)$ be the $n-1$ dimensional sphere centered at $x\in\mathbb{R}^{n}$
with radius $r>0$. The reciprocal of its surface area is denoted
by $c_{r}$. We have 
\[
c_{r}=\frac{1}{2}\,(n=1),\qquad c_{r}=\frac{\Gamma(n/2)r^{n-1}}{2\pi^{n/2}}\,(n\ge2).
\]
The spherical mean value operator $M_{r}$ is defined by 
\[
M_{r}u(x)=\begin{cases}
{\displaystyle \frac{1}{2}\left\{ u(x-r)+u(x+r)\right\} } & (n=1),\\
{\displaystyle c_{r}\int_{S(x,r)}u(y)\,dS_{x,r}(y)} & (n\ge2),
\end{cases}
\]
where $dS_{x,r}$ is the surface area measure of $S(x,r)$. Let $\delta_{S(0,r)}$
be the distribution defined by 
\begin{align*}
\delta_{S(0,r)}\colon\mathcal{\mathcal{C}}^{\infty}(\mathbb{R}^{n}) & \to\mathbb{C},\\
u(x) & \mapsto M_{r}u(0).
\end{align*}
Then we have 
\[
M_{r}u(x)=\delta_{S(0,r)}*u(x),
\]
where $*$ denotes convolution.

Let $n=n_{y}$ be the outer unit normal of $S(x,r)$ at $y\in S(x,r)$.
We introduce the Neumann version of the spherical mean value operator
and its generalization by 
\begin{align*}
M_{r}^{(\ell)}u(x) & =c_{r}\int_{S(x,r)}\left(\frac{\partial}{\partial n_{y}}\right)^{\ell}u(y)\,dS_{x,r}(y)\\
 & =c_{1}\int_{S(0,1)}\frac{\partial^{\ell}}{\partial r^{\ell}}u(x+r\omega)\,dS_{0,1}(\omega)
\end{align*}
for a smooth function $u(x)$ on $\mathbb{R}^{n}$ and a non-negative
integer $\ell$. It is trivial that $M_{r}^{(0)}=M_{r}$.

The Fourier transform of $u$ is defined by $\hat{u}(\xi)=u_{x}(e^{-i\langle x,\xi\rangle})=\langle u(x),e^{-i\langle x,\xi\rangle}\rangle$.
In some cases, it is denoted by $(u)^{\wedge}$.
\begin{thm}
We have
\begin{equation}
M_{r}^{(\ell)}u(x)=\frac{\partial^{\ell}\delta_{S(0,r)}}{\partial r^{\ell}}*u(x)\label{eq:Nrconvol}
\end{equation}
for a smooth function $u(x)$ on $\mathbb{R}^{n}$ and a non-negative
integer $\ell$.
\end{thm}

\begin{proof}
\begin{align*}
M_{r}^{(\ell)}u(x) & =\frac{\partial^{\ell}}{\partial r^{\ell}}c_{1}\int_{S(0,1)}u(x+r\omega)dS_{1}(\omega)\\
 & =\frac{\partial^{\ell}}{\partial r^{\ell}}(M_{r}u)(x)=\frac{\partial^{\ell}}{\partial r^{\ell}}\left[\delta_{S(0,r)}*u(x)\right]=\frac{\partial^{\ell}\delta_{S(0,r)}}{\partial r^{\ell}}*u(x)\qedhere
\end{align*}
\end{proof}
\begin{rem}
The convolution operator $(\partial^{\ell}\delta_{S(0,r)}/\partial r^{\ell})*$
is well-defined on $\mathcal{D}'(\mathbb{R}^{n})$. We will study
$(\partial^{\ell}\delta_{S(0,r)}/\partial r^{\ell})*$ as endomorphisms
on $\mathcal{C}^{\infty}(\mathbb{R}^{n})$, $\mathcal{D}'(\mathbb{R}^{n})$,
$\mathcal{E}'(\mathbb{R}^{n})$ and $\mathcal{\mathcal{C}}_{0}^{\infty}(\mathbb{R}^{n})$.
\end{rem}

\section{Invertibility}
\begin{thm}[{\cite[Theorem 16.3.9, 16.3.10]{Hormander2}}]
\label{thm:invertible}For $u\in\mathcal{E}'(\mathbb{R}^{n})$, the
following statements are equivalent.

(i) There is a constant $A>0$ such that we have
\[
\sup\left\{ |\hat{u}(\zeta)|\,;\,\zeta\in\mathbb{C}^{n},|\zeta-\xi|<A\log(2+|\xi|)\right\} >(A+|\xi|)^{-A}
\]
for any $\xi\in\mathbb{R}^{n}$. 

(ii) If $w\in\mathcal{E}'(\mathbb{R}^{n})$ and $\hat{w}/\hat{u}$
is an holomorphic function, then $\hat{w}/\hat{u}$ is the Fourier
transform of a distribution in $\mathcal{E}'(\mathbb{R}^{n})$. 

(iii) If $v\in\mathcal{E}'(\mathbb{R}^{n})$ satisfies $u*v\in\mathcal{\mathcal{C}}^{\infty}(\mathbb{R}^{n})$,
then $v\in\mathcal{\mathcal{C}}^{\infty}(\mathbb{R}^{n})$.
\end{thm}

\begin{defn}[{\cite[Definition 16.3.12]{Hormander2}}]
\label{def:invertible} We say that an element $u$ of $\mathcal{E}'(\mathbb{R}^{n})$
is invertible if it satisfies the conditions in Theorem \ref{thm:invertible}.
\end{defn}

\begin{rem}
\label{rem:invertibility} The statements (a) and (b) below are equivalent.
If they are true, then so is (i) in Theorem \ref{thm:invertible}.
In other words, (a) and (b) are sufficient conditions for $u$ to
be invertible. 

(a) There is a constant $A>0$ such that we have
\[
\sup\left\{ |\hat{u}(\eta)|\,;\,\eta\in\mathbb{R}^{n},|\eta-\xi|<A\log(2+|\xi|)\right\} >(A+|\xi|)^{-A}
\]
for any $\xi\in\mathbb{R}^{n}$. 

(b) There are constants $A,B>0$ such that we have
\[
\sup\left\{ |\hat{u}(\eta)|\,;\,\eta\in\mathbb{R}^{n},|\eta-\xi|<A\log(2+|\xi|)\right\} >(A+|\xi|)^{-A}
\]
for any $\xi\in\mathbb{R}^{n}$ satisfying $|\xi|>B$. 
\end{rem}

We employ the normalized Bessel function $j_{\nu}(z)$ defined by
\[
j_{\nu}(z)=\Gamma(\nu+1)\left(\frac{2}{z}\right)^{\nu}J_{v}(z),\quad\nu>-1,
\]
where $J_{\nu}(z)$ is the usual Bessel function of the first kind
of order $\nu$. The advantage of $j_{\nu}(z)$ over $J_{\nu}(z)$
is the fact that the former is an even entire function. Indeed, we
have 
\[
j_{\nu}(z)=\sum_{k=0}^{\infty}\frac{(-1)^{k}\Gamma(\nu+1)}{k!\Gamma(k+\nu+1)}\left(\frac{z}{2}\right)^{2k},\,z\in\mathbb{C}.
\]
Recall that all the zeros of $J_{\nu}(z)$ are real and simple except
possibly at $z=0$. Therefore all the zeros of $j_{\nu}(z)$ are real
and simple and they appear in pairs of the form $\pm x_{j}\in\mathbb{R}\setminus\left\{ 0\right\} $. 
\begin{thm}
\label{thm:normalBesselderiv}We have

\begin{align}
j_{\nu}'(z) & =-\frac{zj_{\nu+1}(z)}{2(\nu+1)}.\label{eq:normalBesselderiv}
\end{align}
For $\nu>-1$ and $\ell\ge0$ fixed, there exist constants $C_{\ell,k}=C_{\ell,k}^{(\nu)}$
such that 
\begin{equation}
j_{\nu}^{(\ell)}(z)=\begin{cases}
{\displaystyle \sum_{k=\ell/2}^{\ell}C_{\ell,k}z^{2k-\ell}j_{\nu+k}(z)} & (\ell:\text{even}),\\
{\displaystyle \sum_{k=(\ell+1)/2}^{\ell}C_{\ell,k}z^{2k-\ell}j_{\nu+k}(z)} & (\ell:\text{odd}).
\end{cases}\label{eq:normalBesselderiv ell}
\end{equation}
We have 
\[
C_{\ell,k}=C_{\ell,k}^{(\nu)}=(-1)^{k}\frac{(2k-\ell+1)_{2(\ell-k)}}{2^{\ell}(\ell-k)!(\nu+1)_{k}},
\]
where $\ell/2\le k\le\ell$ if $\ell$ is even and $(\ell+1)/2\le k\le\ell$
if $\ell$ is odd. Here $(a)_{p}$ is the rising factorial: $(a)_{0}=1,\,(a)_{p}=a(a+1)\cdots(a+p-1)$.
\begin{proof}
The well-known formula $\left(z^{-\nu}J_{\nu}(z)\right)'=-z^{-\nu}J_{\nu+1}(z)$
(\cite[10.6.6]{dlmf}) gives 
\[
\left(z^{-\nu}\frac{z^{\nu}j_{\nu}(z)}{2^{\nu}\Gamma(\nu+1)}\right)'=-z^{-\nu}\frac{z^{\nu+1}j_{\nu+1}(z)}{2^{\nu+1}\Gamma(\nu+2)}
\]
and \eqref{eq:normalBesselderiv} follows. The general formula \eqref{eq:normalBesselderiv ell}
can be proved by induction. We have 
\[
\left\{ z^{2k-\ell}j_{\nu+k}(z)\right\} '=(2k-\ell)z^{2k-(\ell+1)}j_{\nu+k}(z)-\frac{z^{2(k+1)-(\ell+1)}}{2(\nu+k+1)}j_{\nu+(k+1)}(z).
\]
Here notice that $2k-\ell$ vanishes for $k=\ell/2$. If $\ell$ is
even, we have 
\begin{align*}
j_{\nu}^{(\ell+1)}(z)= & \sum_{k=\ell/2+1}^{\ell}C_{\ell,k}(2k-\ell)z^{2k-(\ell+1)}j_{\nu+k}(z)\\
 & \quad-\sum_{k=\ell/2}^{\ell}C_{\ell,k}\frac{z^{2(k+1)-(\ell+1)}}{2(\nu+k+1)}j_{\nu+(k+1)}(z)\\
= & \sum_{k'=(\ell+2)/2}^{\ell}C_{\ell,k'}(2k'-\ell)z^{2k'-(\ell+1)}j_{\nu+k'}(z)\\
 & \quad-\sum_{k'=(\ell+2)/2}^{\ell+1}\frac{C_{\ell,k'-1}}{2(\nu+k')}z^{2k'-(\ell+1)}j_{\nu+k'}(z)\\
= & \sum_{k'=[(\ell+1)+1]/2}^{\ell+1}C_{\ell+1,k'}z^{2k'-(\ell+1)}j_{\nu+k'}(z),
\end{align*}
where 
\[
\begin{cases}
\hspace{-1em} & C_{\ell+1,\ell+1}=-\dfrac{C_{\ell,\ell}}{2(\nu+\ell+1)},\\
\hspace{-1em} & C_{\ell+1,k}=(2k-\ell)C_{\ell,k}-\dfrac{C_{\ell,k-1}}{2(\nu+k)}\quad\left(\dfrac{(\ell+1)+1}{2}\le k\le\ell\right).
\end{cases}
\]
 If $\ell$ is odd, we have 
\begin{align*}
j_{\nu}^{(\ell+1)}(z)= & \sum_{k=(\ell+1)/2}^{\ell}C_{\ell,k}(2k-\ell)z^{2k-(\ell+1)}j_{\nu+k}(z)\\
 & \quad-\sum_{k=(\ell+1)/2}^{\ell}C_{\ell,k}\frac{z^{2(k+1)-(\ell+1)}}{2(\nu+k+1)}j_{\nu+(k+1)}(z)\\
= & \sum_{k'=(\ell+1)/2}^{\ell}C_{\ell,k'}(2k'-\ell)z^{2k'-(\ell+1)}j_{\nu+k'}(z)\\
 & \quad-\sum_{k'=(\ell+1)/2+1}^{\ell+1}\frac{C_{\ell,k'-1}}{2(\nu+k')}z^{2k'-(\ell+1)}j_{\nu+k'}(z)\\
= & \sum_{k'=(\ell+1)/2}^{\ell+1}C_{\ell+1,k'}z^{2k'-(\ell+1)}j_{\nu+k'}(z),
\end{align*}
where 
\[
\begin{cases}
\hspace{-1em} & C_{\ell+1,\ell+1}=-\dfrac{C_{\ell,\ell}}{2(\nu+\ell+1)},\\
\hspace{-1em} & C_{\ell+1,k}=(2k-\ell)C_{\ell,k}-\dfrac{C_{\ell,k-1}}{2(\nu+k)}\quad\left(\dfrac{\ell+1}{2}+1\le k\le\ell\right),\\
\hspace{-1em} & C_{\ell+1,(\ell+1)/2}=C_{\ell,(\ell+1)/2}.
\end{cases}
\]
Combining $C_{0,0}=1$ with the recurrence relation, we get the expression
of $C_{\ell,k}$.
\end{proof}
\end{thm}

\begin{lem}
\label{lem:besselbound}Fix $r>0$ and $n\ge1$. Let $\nu=n/2-1$.

(i) There exist constants $A,B>0$ such that 
\begin{align}
\sup\left\{ \bigl|j_{\nu}^{(\ell)}(r|\eta|)\bigr|\,;\,\eta\in\mathbb{R}^{n},|\eta-\xi|<A\log(2+|\xi|)\right\}  & >(A+|\xi|)^{-A}\label{eq:besselbound}
\end{align}
for any $\xi\in\mathbb{R}^{n}$ with $|\xi|>B$.

(ii) There exists a sequence of real numbers $\bigl\{ a_{m}^{(\nu,\ell)}\bigr\}_{m}$
such that $j_{\nu}^{(\ell)}\bigl(\pm a_{m}^{(\nu,\ell)}\bigr)=0$
and $a_{m}^{(\nu,\ell)}\to\infty$ $(m\to\infty)$.
\end{lem}

\begin{proof}
The function $j_{\mu}(x)\,(\mu>-1)$ has the asymptotic behavior (\cite[10.17.3]{dlmf})
\[
j_{\mu}(x)=\frac{\Gamma(\mu+1)}{\sqrt{\pi}}\left(\frac{2}{x}\right)^{\mu+1/2}\left\{ \cos\left(x-\frac{\mu\pi}{2}-\frac{\pi}{4}\right)+O(x^{-1})\right\} \quad(\mathbb{R}\ni x\to\infty).
\]
In the right-hand side of \eqref{eq:normalBesselderiv ell}, the dominant
term is the one corresponding to $k=\ell$. We have 
\begin{equation}
j_{\nu}^{(\ell)}(x)\sim\frac{\Gamma(\nu+1)C_{\ell,\ell}2^{\ell}}{\sqrt{\pi}}\left(\frac{2}{x}\right)^{\nu+1/2}\,\left\{ \cos\left(x-\frac{\pi(\nu+\ell)}{2}-\frac{\pi}{4}\right)+O(x{}^{-1})\right\} \label{eq:besselderivasymp}
\end{equation}
and the limit supremum of $|j_{\nu}^{(\ell)}(x)|$ on $x>\mathrm{const.}$
decays in the order of $x^{-(\nu+1/2)}$. Therefore the estimate \eqref{eq:besselbound}
holds true if $A$ and $B$ are sufficiently large. Moreover, \eqref{eq:besselderivasymp}
implies that the real valued function $j_{\nu}^{(\ell)}(x)$ for $x>0$
is oscillatory, that is, it changes sign infinitely many times as
$x\rightarrow\infty$.
\end{proof}
\begin{thm}
\label{thm:deltafouriertransform}For any $\xi\in\mathbb{R}^{n}$,
we have 
\begin{align}
\widehat{\delta}_{S(0,r)}(\xi) & =j_{n/2-1}(r|\xi|).\label{eq:deltafourier1}
\end{align}
For $\ell\ge0$, the Fourier transform of $\partial\delta_{S(0,r)}^{\ell}/\partial r^{\ell}$
is 
\begin{align}
\left(\frac{\partial^{\ell}\delta_{S(0,r)}}{\partial r^{\ell}}\right)^{\hspace{-0.2em}\wedge}(\xi) & =|\xi|^{\ell}j_{n/2-1}^{(\ell)}(r|\xi|)\label{eq:deltafourier2}\\
 & =\begin{cases}
{\displaystyle \sum_{k=\ell/2}^{\ell}C_{\ell,k}r^{2k-\ell}|\xi|^{2k}j_{n/2-1+k}(r|\xi|)} & (\ell:\text{even}),\\
{\displaystyle \sum_{k=(\ell+1)/2}^{\ell}C_{\ell,k}r^{2k-\ell}|\xi|^{2k}j_{n/2-1+k}(r|\xi|)} & (\ell:\text{odd}).
\end{cases}\nonumber 
\end{align}
\end{thm}

\begin{proof}
By \cite[Introduction, Lemma 3.6]{Helgason}, 
\begin{align*}
\widehat{\delta}{}_{S(0,1)}(\xi) & =c_{1}(2\pi)^{n/2}\frac{J_{n/2-1}(|\xi|)}{|\xi|^{n/2-1}}=\frac{\Gamma(n/2)}{2\pi^{n/2}}(2\pi)^{n/2}\frac{J_{n/2-1}(|\xi|)}{|\xi|^{n/2-1}}\\
 & =2^{n/2-1}\Gamma(n/2)\frac{J_{n/2-1}(|\xi|)}{|\xi|^{n/2-1}}=j_{n/2-1}(|\xi|).
\end{align*}
For general $r>0$, 
\begin{align*}
\widehat{\delta}_{S(0,r)}(\xi) & =\int_{\mathbb{R}^{n}}e^{-i\langle x,\xi\rangle}\delta_{S(0,r)}(x)\,dx=\int_{\mathbb{R}^{n}}e^{-i\langle x,\xi\rangle}r^{-n}\delta_{S(0,1)}(x/r)\,dx\\
 & =\int_{\mathbb{R}^{n}}e^{-ir\langle y,\xi\rangle}\delta_{S(0,1)}(y)\,dy=\widehat{\delta}{}_{S(0,1)}(r\xi)=j_{n/2-1}(r|\xi|).
\end{align*}
The formula \eqref{eq:deltafourier1} has been proved and \eqref{eq:deltafourier2}
follows.
\end{proof}
\begin{cor}
\label{cor:deltafourier}The Fourier transform of $\partial^{\ell}\delta_{S(0,r)}/\partial r^{\ell}$
is an even entire function and we have 
\begin{align*}
\widehat{\delta}_{S(0,r)}(\zeta) & =j_{n/2-1}\bigl(r\sqrt{\zeta^{2}}\bigr)\\
 & =\Gamma\left(\frac{n}{2}\right)\sum_{k=0}^{\infty}\frac{(-1)^{k}}{k!\Gamma(k+n/2)}\left(\frac{r}{2}\right)^{2k}(\zeta^{2})^{k},\\
\left(\frac{\partial\delta_{S(0,r)}}{\partial r}\right)^{\hspace{-0.2em}\wedge}(\zeta) & =-\frac{r\zeta^{2}}{n}j_{n/2}\bigl(r\sqrt{\zeta^{2}}\bigr)\\
 & =-\frac{r\zeta^{2}}{2}\Gamma\left(\frac{n}{2}\right)\sum_{k=0}^{\infty}\frac{(-1)^{k}}{k!\Gamma(k+n/2+1)}\left(\frac{r}{2}\right)^{2k}(\zeta^{2})^{k},\\
\left(\frac{\partial^{\ell}\delta_{S(0,r)}}{\partial r^{\ell}}\right)^{\hspace{-0.2em}\wedge}(\zeta) & =\begin{cases}
{\displaystyle \sum_{k=\ell/2}^{\ell}C_{\ell,k}r^{2k-\ell}(\zeta^{2})^{k}j_{n/2-1+k}\bigl(r\sqrt{\zeta^{2}}\bigr)} & (\ell:\text{even}),\\
{\displaystyle \sum_{k=(\ell+1)/2}^{\ell}C_{\ell,k}r^{2k-\ell}(\zeta^{2})^{k}j_{n/2-1+k}\bigl(r\sqrt{\zeta^{2}}\bigr)} & (\ell:\text{odd}).
\end{cases}
\end{align*}
Here $\zeta=(\zeta_{1},\dots,\zeta_{n})\in\mathbb{C}^{n},\zeta^{2}=\zeta_{1}^{2}+\dots+\zeta_{n}^{2}.$
Notice that $j_{\nu}(\cdot)$ is an even function and $j_{\nu}\bigl(r\sqrt{\zeta^{2}}\bigr)$
is well-defined. 
\end{cor}

\begin{prop}
\label{prop:invertibility}The distribution $\partial^{\ell}\delta_{S(0,r)}/\partial r^{\ell}\,(\ell\ge0)$
is invertible.
\end{prop}

\begin{proof}
The proposition follows from Remark \ref{rem:invertibility} (b),
Lemma \ref{lem:besselbound} and Theorem \ref{thm:deltafouriertransform}.
\end{proof}

\section{Surjectivity on the space of smooth functions}

We introduce the notion of $\mu$-convexity for supports of a pair
of open sets $(X_{1},X_{2})$. When $\mu\in\mathcal{E}'(\mathbb{R}^{n})$,
we set 
\[
\check{\mu}(\phi)=\mu(\check{\phi}),\,\check{\phi}(x)=\phi(-x),
\]
 where $\phi\in\mathcal{\mathcal{C}}_{0}^{\infty}(\mathbb{R}^{n})$
is a test function. In some cases, $\check{\mu}$ is denoted by $(\mu)^{\vee}$. 
\begin{defn}[{\cite[Definition 16.5.4]{Hormander2}}]
 Assume $\mu\in\mathcal{E}'(\mathbb{R}^{n})$. Let $X_{1}$ and $X_{2}$
be non-empty open subsets of $\mathbb{R}^{n}$ satisfying $X_{2}-\mathrm{supp\,}\mu\subset X_{1}$.
We say that $(X_{1},X_{2})$ is $\mu$-convex for supports if for
every compact set $K_{1}\subset X_{1}$ one can find a compact set
$K_{2}\subset X_{2}$ such that $\mathrm{supp\,}v\subset K_{2}$ if
$v\in\mathcal{\mathcal{C}}_{0}^{\infty}(X_{2})$ and $\mathrm{supp\,}\check{\mu}*v\subset K_{1}$.
\end{defn}

We will need the case of $X_{1}=X_{2}=\mathbb{R}^{n}$ only. The condition
$X_{2}-\mathrm{supp\,}\mu\subset X_{1}$ is trivial in that case.
\begin{thm}[{\cite[Theorem 16.5.7]{Hormander2}}]
\label{thm:surjectivity} Assume $\mu\in\mathcal{E}'(\mathbb{R}^{n})$.
Let $X_{1}$ and $X_{2}$ be non-empty open subsets of $\mathbb{R}^{n}$
satisfying $X_{2}-\mathrm{supp\,}\mu\subset X_{1}$. Then the following
two statements are equivalent.

(i) The convolution operator $\mu*\colon\mathcal{\mathcal{C}}^{\infty}(X_{1})\to\mathcal{\mathcal{C}}^{\infty}(X_{2})$
is surjective.

(ii) The distribution $\mu$ is invertible and the pair $(X_{1},X_{2})$
is $\mu$-convex for supports.
\end{thm}

\begin{prop}
\label{prop:delta convex}The pair $(\mathbb{R}^{n},\mathbb{R}^{n})$
is $\partial^{\ell}\delta_{S(0,r)}/\partial r^{\ell}$-convex for
supports for $\ell\ge0$.
\end{prop}

\begin{proof}
Recall that 
\[
\mathrm{ch}\mathrm{\,\mathrm{supp\,}}u_{1}*u_{2}=\mathrm{ch}\mathrm{\,\mathrm{supp\,}}u_{1}+\mathrm{ch}\mathrm{\,\mathrm{supp\,}}u_{2}
\]
holds for any $u_{1},u_{2}\in\mathcal{E}'(\mathbb{R}^{n})$ (\cite[Theorem 4.3.3]{Hormander1}),
where $\mathrm{ch}$ denotes the convex hull. It implies $\mathrm{ch}\mathrm{\,\mathrm{supp\,}}u_{1}\subset\mathrm{ch}\mathrm{\,\mathrm{supp\,}}u_{1}*u_{2}-\mathrm{ch}\mathrm{\,\mathrm{supp\,}}u_{2}$. 

We have 
\[
\mathrm{supp\,}(\partial^{\ell}\delta_{S(0,r)}/\partial r^{\ell})^{\vee}=\mathrm{supp\,}\partial^{\ell}\delta_{S(0,r)}/\partial r^{\ell}=\left\{ x\in\mathbb{R}^{n}\,;\,|x|=r\right\} .
\]
Let $K_{1}\subset\mathbb{R}^{n}$ be an arbitrary compact set and
assume that $v\in\mathcal{\mathcal{C}}_{0}^{\infty}(\mathbb{R}^{n})$
satisfies $\mathrm{supp\,}(\partial^{\ell}\delta_{S(0,r)}/\partial r^{\ell})^{\vee}*v\subset K_{1}$.
We have
\begin{align*}
\mathrm{supp}\,v & \subset\mathrm{ch}\mathrm{\,supp\,}(\partial^{\ell}\delta_{S(0,r)}/\partial r^{\ell})^{\vee}*v-\mathrm{ch}\mathrm{\,supp\,}(\partial^{\ell}\delta_{S(0,r)}/\partial r^{\ell})^{\vee}\\
 & \subset\mathrm{ch}\mathrm{\,}K_{1}-\left\{ x\in\mathbb{R}^{n}\,;\,|x|\le r\right\} .
\end{align*}
Recall that the convex hull of a compact subset in $\mathbb{R}^{n}$
is compact. Therefore $K_{2}=\mathrm{ch}\mathrm{\,}K_{1}-\left\{ x\in\mathbb{R}^{n}\,;\,|x|\le r\right\} $
is a compact set independent of $v$.  
\end{proof}
\begin{thm}
Let $r>0$. The convolution operator $\partial^{\ell}\delta_{S(0,r)}/\partial r^{\ell}*\colon\mathcal{\mathcal{C}}^{\infty}(\mathbb{R}^{n})\to\mathcal{\mathcal{C}}^{\infty}(\mathbb{R}^{n})$
is surjective.
\end{thm}

\begin{proof}
Apply Theorem \ref{thm:surjectivity} and Propositions \ref{prop:invertibility}
and \ref{prop:delta convex}.
\end{proof}
We conclude this section by showing that $\partial^{\ell}\delta_{S(0,r)}/\partial r^{\ell}*\colon\mathcal{\mathcal{C}}^{\infty}(\mathbb{R}^{n})\to\mathcal{\mathcal{C}}^{\infty}(\mathbb{R}^{n})$
is not injective for any $\ell\ge0$. 
\begin{thm}
The kernel of the operator  $\partial^{\ell}\delta_{S(0,r)}/\partial r^{\ell}*\colon\mathcal{\mathcal{C}}^{\infty}(\mathbb{R}^{n})\to\mathcal{\mathcal{C}}^{\infty}(\mathbb{R}^{n})$
is infinite-dimensional for any $\ell\ge0$.
\end{thm}

\begin{proof}
Set 
\[
d_{\ell}(\zeta_{1})=\left(\frac{\partial^{\ell}\delta_{S(0,r)}}{\partial r^{\ell}}\right)^{\hspace{-0.2em}\wedge}(\zeta_{1},0,0\dots,0).
\]
Then by Lemma \ref{lem:besselbound} (ii), $d_{\ell}(\zeta_{1})$
admits infinitely many (real) zeros and so does $(\partial^{\ell}\delta_{S(0,r)}/\partial r^{\ell})^{\wedge}(\zeta)$.
Since $(\partial^{\ell}\delta_{S(0,r)}/\partial r^{\ell})^{\wedge}(\zeta)$
is a function of $\zeta^{2}$, there exists a countably infinite subset
$S$ of $\mathbb{C}$ such that $(\partial^{\ell}\delta_{S(0,r)}/\partial r^{\ell})^{\wedge}(\zeta)$
vanishes on $T=\left\{ \zeta;\,\zeta^{2}\in S\right\} .$ 

On the other hand, we have 
\begin{align*}
\exp(\pm i\langle\zeta,x\rangle)*\frac{\partial^{\ell}\delta_{S(0,r)}}{\partial r^{\ell}}(x) & =\int_{\mathbb{R}^{n}}\frac{\partial^{\ell}\delta_{S(0,r)}}{\partial r^{\ell}}(y)\exp(\pm i\langle\zeta,x-y\rangle)\,dy\\
 & =\left(\frac{\partial^{\ell}\delta_{S(0,r)}}{\partial r^{\ell}}\right)^{\hspace{-0.2em}\wedge}(\zeta)\exp(\pm i\langle\zeta,x\rangle)
\end{align*}
and $\exp(\pm i\langle\zeta,x\rangle)$ belong to the kernel if $\zeta\in T$.
It implies the infinite-dimensionality of the kernel. By integrating
over $T$ with respect to measures, one can construct a large variety
of homogeneous solutions to $\partial^{\ell}\delta_{S(0,r)}/\partial r^{\ell}*u=0$. 

Notice that in the particular cases $\ell=0,1$, we can get a concrete
description of the location of those zeros. Indeed, they correspond
to the zeros of the Bessel function. If $\pm r\sqrt{\zeta^{2}}$ are
zeros of $j_{n/2-1}(\cdot)$, then $\exp(\pm i\langle\zeta,x\rangle)$
belong to the kernel of $\delta_{S(0,r)}*$. If $\pm r\sqrt{\zeta^{2}}$
are zeros of $j_{n/2}(\cdot)$, then $\exp(\pm i\langle\zeta,x\rangle)$
belong to the kernel of $\partial\delta_{S(0,r)}/\partial r*$.
\end{proof}
\begin{rem}
There are elements of the kernel that have very simple expressions.
If $a_{m}\in\mathbb{R\,}(m\ge0)$ satisfy $j_{n/2-1}^{(\ell)}(\pm ra_{m})=0$,
we have $d_{\ell}(\pm a_{m})=0$ and 
\begin{align*}
 & (\partial^{\ell}\delta_{S(0,r)}/\partial r^{\ell})^{\wedge}(\pm a_{m},0,0\dots,0)=(\partial^{\ell}\delta_{S(0,r)}/\partial r^{\ell})^{\wedge}(0,\pm a_{m},0,\dots,0)\\
 & =\dots=(\partial^{\ell}\delta_{S(0,r)}/\partial r^{\ell})^{\wedge}(0,\dots,0,\pm a_{m})=0.
\end{align*}
 Then $\exp(\pm ia_{m}x_{j})$ belong to the kernel for any $j=1,\dots,n$.
\end{rem}

\section{Surjectivity on the space of distributions}
\begin{defn}[{\cite[Definition 16.5.13]{Hormander2}}]
 Assume $\mu\in\mathcal{E}'(\mathbb{R}^{n})$. Let $X_{1},X_{2}$
be non-empty open subsets of $\mathbb{R}^{n}$ satisfying $X_{2}-\mathrm{sing\,supp\,}\mu\subset X_{1}$.
We say that $(X_{1},X_{2})$ is $\mu$-convex for singular supports
if for every compact set $K_{1}\subset X_{1}$ one can find a compact
set $K_{2}\subset X_{2}$ such that $\mathrm{sing\,supp\,}v\subset K_{2}$
if $v\in\mathcal{\mathcal{E}}'(X_{2})$ and $\mathrm{sing\,supp\,}\check{\mu}*v\subset K_{1}$.
\end{defn}

We recall two important facts.
\begin{thm}[{\cite[Corollary 16.5.19]{Hormander2}}]
\label{thm:HormanderCor16.5.19}Assume $\mu\in\mathcal{E}'(\mathbb{R}^{n})$.
Let $X_{1}$ and $X_{2}$ be non-empty open subsets of $\mathbb{R}^{n}$.
Assume $X_{2}-\mathrm{supp\,}\mu\subset X_{1}$. Then $\mu*\mathcal{D}'(X_{1})=\mathcal{D}'(X_{2})$
if and only if $\mu$ is invertible and $(X_{1},X_{2})$ is $\mu$-convex
for supports and singular supports.
\end{thm}

\,
\begin{thm}[{\cite[Corollary 16.3.15]{Hormander2}}]
\label{thm:Corollary 16.3.5} Assume that $u\in\mathcal{E}'(\mathbb{R}^{n})$
is invertible. Then we have 
\[
\mathrm{ch}\,\mathrm{sing\,supp\,}v\subset\mathrm{ch}\,\mathrm{sing\,supp\,}(u*v)-\mathrm{ch}\,\mathrm{sing\,supp\,}u,\quad v\in\mathcal{E}'(\mathbb{R}^{n}).
\]
\end{thm}

The following is an analogue of Proposition \ref{prop:delta convex}.
\begin{prop}
\label{prop:delta convexity for singsupp}The pair $(\mathbb{R}^{n},\mathbb{R}^{n})$
is $\partial^{\ell}\delta_{S(0,r)}/\partial r^{\ell}$-convex for
singular supports for $\ell\ge0$.
\end{prop}

\begin{proof}
The proof is almost the same as that of Proposition \ref{prop:delta convex}.
We can use Proposition \ref{prop:invertibility} and Theorem \ref{thm:Corollary 16.3.5}.
\end{proof}
Our main result in this section is the following.
\begin{thm}
The convolution operator $\partial^{\ell}\delta_{S(0,r)}/\partial r^{\ell}*\colon\mathcal{D}'(\mathbb{R}^{n})\to\mathcal{D}'(\mathbb{R}^{n})$
is surjective.
\end{thm}

\begin{proof}
Apply Theorem \ref{thm:HormanderCor16.5.19} and Propositions \ref{prop:invertibility},
\ref{prop:delta convex} and \ref{prop:delta convexity for singsupp}.
\end{proof}

\section{Compactly supported distributions}

In this and the following sections, we restrict our consideration
to the cases of $\ell=0,1$ and $n\ge2$. 
\begin{lem}
\label{lem:implicitfunctiontheorem}Let $a(\zeta)$ and $b(\zeta)$
be holomorphic functions on an open subset $U$ of $\mathbb{C}^{n}$.
Set $Z=\left\{ \zeta\in U;\,a(\zeta)=0\right\} $ and assume $da(\zeta)\ne0$
on $Z$. If $b(\zeta)$ vanishes on $Z$, then $b(\zeta)/a(\zeta)$
is holomorphic in $U$. 
\end{lem}

\begin{proof}
We have only to prove the claim in a neighborhood of each point of
$Z$. Let $p\in Z$. In a neighborhood $V$ of $p$, there is a new
system of coordinates $(z_{1},\dots,z_{n})$ such that $a(\zeta)=z_{1}$,
$V\cap Z=\{z_{1}=0\}$. Then the function $b$ can be divided by $z_{1}$. 
\end{proof}
\begin{prop}
\label{prop:f_3 division}Assume $n\ge2$. Set $f_{1}(\zeta)=j_{n/2}\bigl(r\sqrt{\zeta^{2}}\bigr)$,
$f_{2}(\zeta)=\zeta^{2}$. Let $Z=\left\{ \zeta\in\mathbb{C}^{n};\,f_{1}(\zeta)f_{2}(\zeta)=0\right\} $.
If $g(\zeta)$ is holomorphic in a neighborhood of $Z$ and $g(\zeta)=0$
on $Z$, then $g(\zeta)/[f_{1}(\zeta)f_{2}(\zeta)]$ is holomorphic. 
\end{prop}

\begin{proof}
Set $Z_{j}=\{\zeta\in\mathbb{C}^{n};\,f_{j}(\zeta)=0\}\,(j=1,2)$.
Then 
\[
Z=Z_{1}\cup Z_{2},\,Z_{1}\cap Z_{2}=\emptyset,\,0\not\in Z_{1},\,0\in Z_{2}.
\]

By Lemma \ref{lem:implicitfunctiontheorem}, $g(\zeta)/f_{1}(\zeta)$
is holomorphic in a neighborhood of $Z$. By Lemma \ref{lem:implicitfunctiontheorem}
again, $g(\zeta)/[f_{1}(\zeta)f_{2}(\zeta)]$ is holomorphic in a
neighborhood of $Z\setminus\left\{ 0\right\} $. Therefore $0$ is
an isolated singularity. By Hartogs's extension theorem, it is removable.
\end{proof}
\begin{thm}
\label{thm:range E'}Assume $\ell=0,1$ and $n\ge2$. Then we have
the following characterization of the range of the endomorphism $\partial^{\ell}\delta_{S(0,r)}/\partial r^{\ell}*$
on $\mathcal{E}'(\mathbb{R}^{n})$. 
\begin{align}
 & \mathrm{range}\left(\partial^{\ell}\delta_{S(0,r)}/\partial r^{\ell}*\colon\mathcal{E}'(\mathbb{R}^{n})\to\mathcal{E}'(\mathbb{R}^{n})\right)\label{eq:rangeformula}\\
 & =\left\{ w\in\mathcal{E}'(\mathbb{R}^{n});\,\hat{w}(\zeta)=0\;\text{if}\;(\partial^{\ell}\delta_{S(0,r)}/\partial r^{\ell})^{\wedge}(\zeta)=0\right\} .\nonumber 
\end{align}
 The case of $\ell=0,n\ge1$ was proved in \cite[Theorem 2.17]{Lim}.
\end{thm}

\begin{proof}
Here we show the case of $\ell=1$. The proof of the case of $\ell=0$
is easier. Let the left- and right-hand sides of \eqref{eq:rangeformula}
be $R$ and $W$ respectively. The Fourier transform of $\partial\delta_{S(0,r)}/\partial r*v$
is $(\partial\delta_{S(0,r)}/\partial r)^{\wedge}(\zeta)\hat{v}(\zeta)$
and the inclusion $R\subset W$ follows immediately. 

Next we show $R\supset W$. If $w$ belongs to $W$, then $\hat{w}(\zeta)/(\partial\delta_{S(0,r)}/\partial r)^{\wedge}(\zeta)$
is entire by Corollary \ref{cor:deltafourier} and Proposition \ref{prop:f_3 division}.
By Proposition \ref{prop:invertibility} and Theorem \ref{thm:invertible},
$\hat{w}(\zeta)/(\partial\delta_{S(0,r)}/\partial r)^{\wedge}(\zeta)=\hat{v}(\zeta)$
for some distribution $v$ in $\mathcal{E}'(\mathbb{R}^{n})$. Then
$\hat{w}(\zeta)=(\partial\delta_{S(0,r)}/\partial r)^{\wedge}(\zeta)\hat{v}(\zeta)$
and $w=\partial\delta_{S(0,r)}/\partial r*v$. 
\end{proof}

\section{Compactly supported smooth functions}
\begin{thm}
Assume $\ell=0,1$ and $n\ge2$. Then we have the following characterization
of the range of the endomorphism $\partial^{\ell}\delta_{S(0,r)}/\partial r^{\ell}*$
on $\mathcal{\mathcal{C}}_{0}^{\infty}(\mathbb{R}^{n})$. 
\begin{align}
 & \mathrm{range}\left(\partial^{\ell}\delta_{S(0,r)}/\partial r^{\ell}*\colon\mathcal{\mathcal{C}}_{0}^{\infty}(\mathbb{R}^{n})\to\mathcal{\mathcal{C}}_{0}^{\infty}(\mathbb{R}^{n})\right)\label{eq:rangeformula-1}\\
 & =\left\{ w\in\mathcal{\mathcal{C}}_{0}^{\infty}(\mathbb{R}^{n});\,\hat{w}(\zeta)=0\;\text{if}\;(\partial^{\ell}\delta_{S(0,r)}/\partial r^{\ell})^{\wedge}(\zeta)=0\right\} .\nonumber 
\end{align}
 The case of $\ell=0,n\ge1$ was proved in \cite[Theorem 2.23]{Lim}.
\end{thm}

\begin{proof}
Here we show the case of $\ell=1$. The proof of the case of $\ell=0$
is easier. Let the left- and right-hand sides be $R'$ and $W'$ respectively.
The inclusion $R'\subset W'$ is trivial. 

Next we show $R'\supset W'$. If $w$ belongs to $W'$, then $w=\partial\delta_{S(0,r)}/\partial r*v$
for some distribution $v$ in $\mathcal{E}'(\mathbb{R}^{n})$ by Theorem
\ref{thm:range E'}. Since $w\in\mathcal{\mathcal{C}}_{0}^{\infty}(\mathbb{R}^{n})\subset\mathcal{\mathcal{C}}^{\infty}(\mathbb{R}^{n})$,
the invertibility of $\partial\delta_{S(0,r)}/\partial r$ and Theorem
\ref{thm:invertible} (iii) imply that $v\in\mathcal{\mathcal{C}}^{\infty}(\mathbb{R}^{n})$.
Hence $v\in\mathcal{E}'(\mathbb{R}^{n})\cap\mathcal{\mathcal{C}}^{\infty}(\mathbb{R}^{n})=\mathcal{\mathcal{C}}_{0}^{\infty}(\mathbb{R}^{n})$
and $w\in R'$. 
\end{proof}

\end{document}